\documentclass[final,3p]{elsarticle}
 \usepackage{graphics}
 \usepackage{graphicx}
 \usepackage{epsfig}
\usepackage{amssymb}
 \usepackage{amsthm}
 \usepackage{lineno}
 \usepackage{amsmath}
   \numberwithin{equation}{section}
\usepackage{mathrsfs}

\newtheorem{thm}{Theorem}[section]

\newtheorem{lem}[thm]{Lemma}

\biboptions{sort&compress,square}
\journal{Bulletin Des Sciences Math\'{e}matiques}
\begin{document}
\begin{frontmatter}
\author[rvt1]{Sining Wei}
\ead{weisn835@nenu.edu.cn}
\author[rvt2]{Yong Wang\corref{cor2}}
\ead{wangy581@nenu.edu.cn}
\cortext[cor2]{Corresponding author.}

\address[rvt1]{School of Data Science and Artificial Intelligence, Dongbei University of Finance and Economics,\\
Dalian, 116025, P.R.China}
\address[rvt2]{School of Mathematics and Statistics, Northeast Normal University,
Changchun, 130024, China}

\title{Transformation operators and the Kastler-Kalau-Walze type theorems on 4-dimensional manifolds}
\begin{abstract}
In this paper, we compute the
lower-dimensional volume Vol(1,1) about transformation operators for 4-dimensional spin manifolds with boundary and we also get
the Kastler-Kalau-Walze type theorem about transformation operators on 4-dimensional compact manifolds with boundary.
\end{abstract}

\end{frontmatter}
\section{Introduction}
\label{1}
Numerous geometers have delved into the realm of noncommutative residues, recognizing their profound significance in the realm of noncommutative geometry. In \cite{Gu,Wo}, the authors highlighted the pivotal role of noncommutative residues. Initially discovered by Adler \cite{MA} in the context of one-dimensional manifolds, its significance became more apparent through Wodzicki's work \cite{Wo} on arbitrary closed compact $n$-dimensional manifolds, employing the theory of zeta functions of elliptic pseudo-differential operators. Additionally, Connes, in his work, utilized noncommutative residues to derive an analogy to the conformal 4-dimensional Polyakov action, as elucidated in \cite{Co1}. Furthermore, Connes demonstrated the equivalence between the noncommutative residue on a compact manifold $M$ and Dixmier's trace on pseudo-differential operators of order $-{\rm {dim}}M$, as expounded in \cite{Co2}.
Connes also put forth the notion that the noncommutative residue arising from the square of the inverse of the Dirac operator was proportioned to the Einstein-Hilbert action.
Kastler \cite{Ka} provided a rigorous proof of this theorem, employing a meticulous approach. Kalau and Walze independently corroborated this theorem within the framework of normal coordinate systems in their work \cite{KW}.
Ackermann demonstrated that the Wodzicki residue of the square of the inverse of the Dirac operator, denoted as ${\rm Wres}(D^{-2})$, essentially corresponds to the second coefficient in the heat kernel expansion of $D^{2}$ in \cite{Ac}.

In \cite{RP}, Ponge introduced a method for defining lower-dimensional volumes of Riemannian manifolds utilizing the Wodzicki residue. Fedosov et al. established a noncommutative residue within Boutet de Monvel's algebra and demonstrated its uniqueness as a continuous trace in \cite{FGLS}. Schrohe, in \cite{S}, elucidated the relationship between the Dixmier trace and the noncommutative residue specifically for manifolds with boundary. Wang extended the Kastler-Kalau-Walze theorem to encompass 3 and 4-dimensional spin manifolds with boundary, presenting a theorem of similar nature in \cite{Wa3}. Additionally, in a series of works \cite{Wa4,WJ2,WJ3,WJ4}, Y. Wang and collaborators computed lower-dimensional volumes for spin manifolds of dimensions 5, 6 and 7 with boundary while also deriving analogous Kastler-Kalau-Walze theorems. 
In essence, Wang presents a methodology for exploring the Kastler-Kalau-Walze type theorem in the context of manifolds with boundary.

In \cite{WJ6}, J. Wang and Y. Wang established two distinct Kastler-Kalau-Walze type theorems.
These theorems pertain to conformal perturbations
of twisted Dirac operators and conformal perturbations
of signature operators, facilitated by a vector bundle with a non-unitary connection on four-dimensional manifolds with (respectively without) boundary. In \cite{UE}, Ertem introduced the transformation operator which transformed the twistor spinors to harmonic spinors and this operator is important.

{\bf The primary objective of this study} is to investigate the lower-dimensional volume $\text{Vol}(1,1)$ associated with transformation operators on 4-dimensional spin manifolds with boundary, with a particular focus on establishing Kastler-Kalau-Walze type theorems pertaining to these operators. It is noteworthy that the leading symbol of these transformation operators does not conform to the form $\sqrt{-1}c(\xi)$, which underscores the need for our investigation into the residues of such operators. By exploring the relationship between the transformation operators and the geometric and topological structure of the manifolds, we aim to derive a Kastler-Kalau-Walze-type theorem specifically for 4-dimensional compact manifolds with boundary. Through this research, we hope to contribute to the field of spin manifolds, transformation operators, and volume calculations, extending and deepening existing theories, while providing new insights and directions for future studies.\\

The structure of this paper unfolds as follows: Section 2 revisits fundamental concepts and formulas concerning lower-dimensional volumes of spin manifolds  with boundary. In Section 3, we give the transformation operators and its Lichnerowicz formula. In Sections 4, we present a Kastler-Kalau-Walze type theorems tailored to the realm of transformation operators on 4-dimensional manifolds with boundary. Notably, the pivotal findings of this paper are encapsulated within Theorem \ref{th:311},
Theorem \ref{th:312} and Theorem \ref{999}.

\section{Lower-Dimensional Volumes of Spin Manifolds  with  boundary}
In this section, we shall recall
some basic facts and formulas about Boutet de Monvel's calculus and lower-dimensional volumes of spin manifolds  with boundary.

Let $$ F:L^2({\bf R}_t)\rightarrow L^2({\bf R}_v);~F(u)(v)=\int_{R} e^{-ivt}u(t)dt$$ denote the Fourier transformation and
$\varphi(\overline{{\bf R}^+}) =r^+\varphi({\bf R})$ (similarly define $\varphi(\overline{{\bf R}^-}$)), where $\varphi({\bf R})$
denotes the Schwartz space and
  \begin{equation}
r^{+}:C^\infty ({\bf R})\rightarrow C^\infty (\overline{{\bf R}^+});~ f\rightarrow f|\overline{{\bf R}^+};~
 \overline{{\bf R}^+}=\{x\geq0;x\in {\bf R}\}.
\end{equation}
We define $H^+=F(\varphi(\overline{{\bf R}^+}));~ H^-_0=F(\varphi(\overline{{\bf R}^-}))$ which are orthogonal to each other. We have the following
 property: $h\in H^+~$(respectively~$H^-_0)$ if and only if $h\in C^\infty({\bf R})$ which has an analytic extension to the lower (respectively~upper) complex
half-plane $\{{\rm Im}\xi<0\}~$(respectively~$\{{\rm Im}\xi>0\})$ such that for all nonnegative integer $l$,
 \begin{equation}
\frac{d^{l}h}{d\xi^l}(\xi)\sim\sum^{\infty}_{k=1}\frac{d^l}{d\xi^l}(\frac{c_k}{\xi^k})
\end{equation}
as $|\xi|\rightarrow +\infty,{\rm Im}\xi\leq0~$(respectively~${\rm Im}\xi\geq0)$.

 Let $H'$ be the space of all polynomials and $H^-=H^-_0\bigoplus H';~H=H^+\bigoplus H^-.$ Denote by $\pi^+~$(respectively~$\pi^-)$ the
 projection on $H^+~$(respectively~$H^-)$. For calculations, we take $H=\widetilde H=\{$rational functions having no poles on the real axis$\}$ ($\tilde{H}$
 is a dense set in the topology of $H$). Then on $\tilde{H}$,
 \begin{equation}
\pi^+h(\xi_0)=\frac{1}{2\pi i}\lim_{u\rightarrow 0^{-}}\int_{\Gamma^+}\frac{h(\xi)}{\xi_0+iu-\xi}d\xi,
\end{equation}
where $\Gamma^+$ is a Jordan close curve included ${\rm Im}\xi>0$ surrounding all the singularities of $h$ in the upper half-plane and
$\xi_0\in {\bf R}$. Similarly, define $\pi^{'}$ on $\tilde{H}$,
 \begin{equation}
\pi'h=\frac{1}{2\pi}\int_{\Gamma^+}h(\xi)d\xi.
\end{equation}
So, $\pi'(H^-)=0$. For $h\in H\bigcap L^1(\mathbb{R})$, $\pi'h=\frac{1}{2\pi}\int_{\bf R}h(v)dv$ and for $h\in H^+\bigcap L^1(\mathbb{R})$, $\pi'h=0$.

Denote by $\mathcal{B}$ Boutet de Monvel's algebra.
For a detailed introduction to Boutet de Monvel's algebra see Boutet de
Monvel \cite{LB}, Grubb \cite{GG}, Rempel-Schulze \cite{BS} or Schrohe-Schulze \cite{EB}.
In the following, we will give a review of some basic fact we need.

An operator of order $m\in {\bf Z}$ and type $d$ is a matrix
$$A=\left(\begin{array}{lcr}
  \pi^+P+G  & K  \\
   T  &  S
\end{array}\right):
\begin{array}{cc}
\   C^{\infty}(X,E_1)\\
 \   \bigoplus\\
 \   C^{\infty}(\partial{X},F_1)
\end{array}
\longrightarrow
\begin{array}{cc}
\   C^{\infty}(X,E_2)\\
\   \bigoplus\\
 \   C^{\infty}(\partial{X},F_2)
\end{array},
$$
where $X$ is a manifold with boundary $\partial X$ and
$E_1,E_2~$(respectively~$F_1,F_2)$ are vector bundles over $X~$(respectively~$\partial X
)$.\\~Here,~$P:C^{\infty}_0(\Omega,\overline {E_1})\rightarrow
C^{\infty}(\Omega,\overline {E_2})$ is a classical
pseudodifferential operator of order $m$ on $\Omega$, where
$\Omega$ is an open neighborhood of $X$ and
$\overline{E_i}|X=E_i~(i=1,2)$. Then $P$ has an extension:
$~{\cal{E'}}(\Omega,\overline {E_1})\rightarrow
{\cal{D'}}(\Omega,\overline {E_2})$, where
${\cal{E'}}(\Omega,\overline {E_1})$ and ${\cal{D'}}(\Omega,\overline
{E_2})$ are the dual space of $C^{\infty}(\Omega,\overline
{E_1})$ and $C^{\infty}_0(\Omega,\overline {E_2})$. Let
$e^+:C^{\infty}(X,{E_1})\rightarrow{\cal{E'}}(\Omega,\overline
{E_1})$ denote extension by zero from $X$ to $\Omega$ and
$r^+:{\cal{D'}}(\Omega,\overline{E_2})\rightarrow
{\cal{D'}}(\Omega, {E_2})$ denote the restriction from $\Omega$ to
$X$, then define
$$\pi^+P=r^+Pe^+:C^{\infty}(X,{E_1})\rightarrow {\cal{D'}}(\Omega,
{E_2}).$$
In addition, $P$ is supposed to have the
transmission property; this means that, for all $j,k,\alpha$, the
homogeneous component $p_j$ of order $j$ in the asymptotic
expansion of the
symbol $p$ of $P$ in local coordinates near the boundary satisfies:
$$\partial^k_{x_n}\partial^\alpha_{\xi'}p_j(x',0,0,+1)=
(-1)^{j-|\alpha|}\partial^k_{x_n}\partial^\alpha_{\xi'}p_j(x',0,0,-1),$$
then $\pi^+P$ maps $C^{\infty}(X,{E_1})$ into $C^{\infty}(X,{E_2})$
by Section 2.1 of \cite{Wa5}.

Let $G$, $T$ be respectively the singular Green operator
and the trace operator of order $m$ and type $d$. $K$ is a
potential operator and $S$ is a classical pseudodifferential
operator of order $m$ along the boundary (for detailed definition,
see [11]). Denote by $B^{m,d}$ the collection of all operators of
order $m$
and type $d$,  and $\mathcal{B}$ is the union over all $m$ and $d$.\\
\indent Recall $B^{m,d}$ is a Fr\'{e}chet space. The composition
of the above operator matrices yields a continuous map:
$B^{m,d}\times B^{m',d'}\rightarrow B^{m+m',{\rm max}\{
m'+d,d'\}}.$ Write $$\widetilde{A}=\left(\begin{array}{lcr}
 \pi^+P+G  & K \\
 T  &  \widetilde{S}
\end{array}\right)
\in B^{m,d},
 \widetilde{A}'=\left(\begin{array}{lcr}
\pi^+P'+G'  & K'  \\
 T'  &  \widetilde{S}'
\end{array} \right)
\in B^{m',d'}.$$ The composition $\widetilde{A}\widetilde{A}'$ is obtained by
multiplication of the matrices(for more details see [14]). For
example $\pi^+P\circ G'$ and $G\circ G'$ are singular Green
operators of type $d'$ and
$$\pi^+P\circ\pi^+P'=\pi^+(PP')+L(P,P').$$ Here $PP'$ is the usual
composition of pseudodifferential operators and $L(P,P')$ called
leftover term is a singular Green operator of type $m'+d$. For our case, $P,P'$ are classical pseudo differential operators, in other words $\pi^+P\in \mathcal{B}^{\infty}$ and $\pi^+P'\in \mathcal{B}^{\infty}$ .

In the following, write $\pi^+D^{-1}=\left(\begin{array}{lcr}
  \pi^+D^{-1}  & 0  \\
   0  &  0
\end{array}\right)$.
Let $M$ be a compact manifold with boundary $\partial M$. We assume that the metric $g^{M}$ on $M$ has
the following form near the boundary
 \begin{equation}
 g^{M}=\frac{1}{h(x_{n})}g^{\partial M}+dx _{n}^{2} ,
\end{equation}
where $g^{\partial M}$ is the metric on $\partial M$. Let $U\subset
M$ be a collar neighborhood of $\partial M$ which is diffeomorphic $\partial M\times [0,1)$. By the definition of $h(x_n)\in C^{\infty}([0,1))$
and $h(x_n)>0$, there exists $\tilde{h}\in C^{\infty}\big((-\varepsilon,1)\big)$ such that $\tilde{h}|_{[0,1)}=h$ and $\tilde{h}>0$ for some
sufficiently small $\varepsilon>0$. Then there exists a metric $\hat{g}$ on $\hat{M}=M\bigcup_{\partial M}\partial M\times
(-\varepsilon,0]$ which has the form on $U\bigcup_{\partial M}\partial M\times (-\varepsilon,0 ]$
 \begin{equation}
\hat{g}=\frac{1}{\tilde{h}(x_{n})}g^{\partial M}+dx _{n}^{2} ,
\end{equation}
such that $\hat{g}|_{M}=g$.
We fix a metric $\hat{g}$ on the $\hat{M}$ such that $\hat{g}|_{M}=g$.

Consider the $n-1$-form
$$\sigma(\xi)=\sum\limits^{n}_{j=1}(-1)^{j+1}\xi_{j}d\xi_{1}
\wedge\cdots\wedge\widehat{d\xi_{j}}\wedge\cdots\wedge d\xi_{n},$$
where the hat indicates that the corresponding factor has been omitted.

Restricted $\sigma(\xi)$ to the $n-1$-dimensional unit sphere $|\xi|=1$, $\sigma(\xi)$ gives the volume form on $|\xi|=1$.
Denoting by $|\xi'|=1$ and $\sigma(\xi')$ the $n-2$-dimensional unit sphere and the corresponding $n-2$-form.

Denote by $\mathcal{B}^{\infty}$ the algera of all operators in Boutet de Monvel's calculus (with integral order) and by $\mathcal{B}^{-\infty}$ the ideal of all smoothing operators in $\mathcal{B}^{\infty}$.
Now we recall the main theorem in \cite{FGLS}.

\begin{thm}\label{th:32}{\bf(Fedosov-Golse-Leichtnam-Schrohe)}
 Let $X$ and $\partial X$ be connected, ${\rm dim}X=n\geq3$,
 $A=\left(\begin{array}{lcr}\pi^+P+G &   K \\
T &  S    \end{array}\right)$ $\in \mathcal{B}$, and denote by $p$, $b$ and $s$ the local symbols of $P,~G$ and $S$ respectively.
 Define:
 \begin{eqnarray}
{\rm{\widetilde{Wres}}}(A)&=&\int_X\int_{\bf |\xi|=1}{\rm{tr}}_E\left[p_{-n}(x,\xi)\right]\sigma(\xi)dx \nonumber\\
&&+2\pi\int_ {\partial X}\int_{\bf |\xi'|=1}\left\{{\rm tr}_E\left[({\rm{tr}}b_{-n})(x',\xi')\right]+{\rm{tr}}
_F\left[s_{1-n}(x',\xi')\right]\right\}\sigma(\xi')dx',
\end{eqnarray}
Then

~~ a) ${\rm \widetilde{Wres}}([A,B])=0 $, for any $A,B\in\mathcal{B}$;

~~ b) It is a unique continuous trace on
$\mathcal{B}/\mathcal{B}^{-\infty}$.
\end{thm}

\indent Let $p_{1},p_{2}$ be nonnegative integers and $p_{1}+p_{2}\leq n$. Then by Section 2.1 of \cite{Wa3}, we have\\

\noindent {\bf Definition 2.2.} Lower-dimensional volumes of spin manifolds with boundary  are defined by
\begin{align}\label{a18}
{\rm Vol}^{(p_1,p_2)}_nM:=\widetilde{{\rm Wres}}[\pi^+D^{-p_1}\circ\pi^+D^{-p_2}].
\end{align}
Denote by $\sigma_{l}(A)$ the $l$-order symbol of an operator A. Similar to (2.1.7) in \cite{Wa3}, we have that
\begin{align}\label{a16}
\widetilde{{\rm Wres}}[\pi^+D^{-p_1}\circ\pi^+D^{-p_2}]=\int_M\int_{|\xi|=1}{\rm
trace}_{S(TM)}[\sigma_{-n}(D^{-p_1-p_2})]\sigma(\xi)dx+\int_{\partial
M}\Phi,
\end{align}
where
 \begin{eqnarray}\label{a17}
\Phi&=&\int_{|\xi'|=1}\int^{+\infty}_{-\infty}\sum^{\infty}_{j, k=0}
\sum\frac{(-i)^{|\alpha|+j+k+1}}{\alpha!(j+k+1)!}
 {\rm trace}_{S(TM)}
\Big[\partial^j_{x_n}\partial^\alpha_{\xi'}\partial^k_{\xi_n}
\sigma^+_{r}(D^{-p_1})(x',0,\xi',\xi_n)\nonumber\\
&&\times\partial^\alpha_{x'}\partial^{j+1}_{\xi_n}\partial^k_{x_n}\sigma_{l}
(D^{-p_2})(x',0,\xi',\xi_n)\Big]d\xi_n\sigma(\xi')dx',
\end{eqnarray}
and the sum is taken over $r-k+|\alpha|+\ell-j-1=-n,r\leq-p_{1},\ell\leq-p_{2}$.

Since $[\sigma_{-n}(D^{-p_1-p_2})]|_M$ has the same expression as $\sigma_{-n}(D^{-p_1-p_2})$ in the case of manifolds without
boundary, so locally we can compute the first term by \cite{Ka}, \cite{KW}, \cite{RP}, \cite{Wa3}.

For any fixed point $x_0\in\partial M$, we choose the normal coordinates
$U$ of $x_0$ in $\partial M$ (not in $M$) and compute $\Phi(x_0)$ in the coordinates $\widetilde{U}=U\times [0,1)\subset M$ and with the
metric $\frac{1}{h(x_n)}g^{\partial M}+dx_n^2.$ The dual metric of $g$ on $\widetilde{U}$ is ${h(x_n)}g^{\partial M}+dx_n^2.$  Write
$g_{ij}=g(\frac{\partial}{\partial x_i},\frac{\partial}{\partial x_j});~ g^{ij}=g(dx_i,dx_j)$, then
\begin{equation}
\label{b9}
[g_{ij}]= \left[\begin{array}{lcr}
  \frac{1}{h(x_n)}[g_{ij}^{\partial M}]  & 0  \\
   0  &  1
\end{array}\right];~~~
[g^{ij}]= \left[\begin{array}{lcr}
  h(x_n)[g^{ij}_{\partial M}]  & 0  \\
   0  &  1
\end{array}\right],
\end{equation}
and
\begin{equation}
\label{b10}
\partial_{x_s}g_{ij}^{\partial M}(x_0)=0, 1\leq i,j\leq n-1; ~~~g_{ij}^{TM}(x_0)=\delta_{ij}.
\end{equation}
\indent From \cite{Wa3}, we can get three lemmas.
\begin{lem}{\rm \cite{Wa3}}\label{le:3112}
With the metric $g$ on $M$ near the boundary
\begin{align}
\label{b11}
\partial_{x_j}(|\xi|_{g^M}^2)(x_0)&=\left\{
       \begin{array}{c}
        0,  ~~~~~~~~~~ ~~~~~~~~~~ ~~~~~~~~~~~~~{\rm if }~j<n, \\[2pt]
       h'(0)|\xi'|^{2}_{g^{\partial M}},~~~~~~~~~~~~~~~~~~~~{\rm if }~j=n,
       \end{array}
    \right. \\
\partial_{x_j}[c(\xi)](x_0)&=\left\{
       \begin{array}{c}
      0,  ~~~~~~~~~~ ~~~~~~~~~~ ~~~~~~~~~~~~~{\rm if }~j<n,\\[2pt]
\partial_{x_n}(c(\xi'))(x_{0}), ~~~~~~~~~~~~~~~~~{\rm if }~j=n,
       \end{array}
    \right.
\end{align}
where $\xi=\xi'+\xi_{n}dx_{n}$.
\end{lem}
\begin{lem}{\rm \cite{Wa3}}\label{le:32}With the metric $g$ on $M$ near the boundary
\begin{align}
\label{b12}
\omega_{s,t}(e_i)(x_0)&=\left\{
       \begin{array}{c}
        \omega_{n,i}(e_i)(x_0)=\frac{1}{2}h'(0),  ~~~~~~~~~~ ~~~~~~~~~~~{\rm if }~s=n,t=i,i<n, \\[2pt]
       \omega_{i,n}(e_i)(x_0)=-\frac{1}{2}h'(0),~~~~~~~~~~~~~~~~~~~{\rm if }~s=i,t=n,i<n,\\[2pt]
    \omega_{s,t}(e_i)(x_0)=0,~~~~~~~~~~~~~~~~~~~~~~~~~~~other~cases,~~~~~~~~~
       \end{array}
    \right.
\end{align}
where $(\omega_{s,t})$ denotes the connection matrix of Levi-Civita connection $\nabla^L$.
\end{lem}
\begin{lem}{\rm \cite{Wa3}}\label{lem1} When $i<n,$ then
\begin{align}
\label{b13}
\Gamma_{st}^k(x_0)&=\left\{
       \begin{array}{c}
        \Gamma^n_{ii}(x_0)=\frac{1}{2}h'(0),~~~~~~~~~~ ~~~~~~~~~~~{\rm if }~s=t=i,k=n, \\[2pt]
        \Gamma^i_{ni}(x_0)=-\frac{1}{2}h'(0),~~~~~~~~~~~~~~~~~~~{\rm if }~s=n,t=i,k=i,\\[2pt]
        \Gamma^i_{in}(x_0)=-\frac{1}{2}h'(0),~~~~~~~~~~~~~~~~~~~{\rm if }~s=i,t=n,k=i,\\[2pt]
       \end{array}
    \right.
\end{align}
in other cases, $\Gamma_{st}^i(x_0)=0$.
\end{lem}
\section{Transformation operators and its Lichnerowicz formula}
Let's revisit the definition of the Dirac operator. Consider $M$, an $n$-dimensional ($n\geq 3$)
oriented compact spin Riemannian manifold equipped with a Riemannian metric $g^{M}$ and denote $\nabla^L$ as the Levi-Civita connection associated with $g^{M}$. In local coordinates $\{x_1,\cdots, x_n\}$ and a fixed orthonormal frame $\{e_1,\cdots,e_n\}$, the connection matrix $(\omega_{s,t})$ is defined by
\begin{equation}
\label{a2}
\nabla^L(e_1,\cdots,e_n)= (e_1,\cdots,e_n)(\omega_{s,t}).
\end{equation}
\indent Let $\epsilon (e_j^*)$,~$\iota (e_j^*)$ be the exterior and interior multiplications respectively, where $e_j^*=g^{TM}(e_j,\cdot)$. Let $c(e_i)$ denote the Clifford action.
Write
\begin{equation}\label{a3}
\widehat{c}(e_j)=\epsilon (e_j^* )+\iota
(e_j^*);~~
c(e_j)=\epsilon (e_j^* )-\iota (e_j^* ),
\end{equation}
which satisfies
\begin{align}
\label{a4}
&\widehat{c}(e_i)\widehat{c}(e_j)+\widehat{c}(e_j)\widehat{c}(e_i)=2g^{M}(e_i,e_j);~~\nonumber\\
&c(e_i)c(e_j)+c(e_j)c(e_i)=-2g^{M}(e_i,e_j);~~\nonumber\\
&c(e_i)\widehat{c}(e_j)+\widehat{c}(e_j)c(e_i)=0.
\end{align}

In an $n$-dimensional spin manifold $M$, one can define two different first-order differential operators on spinor fields. The first one is the Dirac operator,
citing \cite{Y}, we present the Dirac operator as follows:
\begin{align}
\label{a5}
&D=\sum^n_{i=1}c(e_i)[e_i-\frac{1}{4}\sum_{s,t}\omega_{s,t}
(e_i)c(e_s)c(e_t)].
\end{align}
And the second one is the Penrose operator written as follows:
\begin{align}\label{z2}
&P_X:=\nabla^{L}_{X}-\frac{1}{n}\widetilde{X}\cdot D
\end{align}
with respect to a vector field $X$ and its metric dual $\widetilde{X}$. The spinor fields that are in the kernel of the Penrose operator are called twistor spinors and they satisfy the following twistor equation:
\begin{align}\label{z3}
\nabla^{L}_{X}{\psi}=\frac{1}{n}\widetilde{X}\cdot D\psi
\end{align}
for a spinor $\psi$(for detailed definition, see \cite{UE}).


Let us consider a function $f$ that satisfies the following conformal Laplace equation in $n$-dimensions
\begin{align}\label{z1}
\Delta f-\frac{n-2}{4(n-1)}Rf=0
\end{align}
and a twistor spinor $\psi$ that satisfies (\ref{z3}). Then, the following operator
\begin{align}\label{z4}
\widetilde{D}=\frac{n-2}{n}fD+c(df).
\end{align}
transforms the twistor spinors to harmonic spinors. Namely, if $\psi$ is a twistor spinor,
then $\widetilde{D}\psi$ satisfies the massless Dirac equation $D\widetilde{D}\psi=0$(for detailed proof, see the formula (54) of \cite{UE}).

In an $4$-dimensional spin manifold $M$, let us consider transformation operators
$$\widetilde{D}=\frac{1}{2}fD+c(df),$$
where $c(df)$ denotes a Clifford action on the manifold $M$ and $f$ is a nonzero smooth function on $M$.
Then, we can express the square of $\widetilde{D}$ as follows:
\begin{align}\label{a6}
\widetilde{D}^2=[\frac{1}{2}fD+c(df)]^2=[\frac{1}{2}D f+\frac{1}{2}c(df)]^2
&=\frac{1}{4}\overline{D}^2f^2,
\end{align}
where $\overline{D}^2=D^2+c(df)Df^{-1}-\frac{1}{f^2}|df|^2$.

Recall the definition of the Dirac operator $D^{2}$ in \cite{Ka}, \cite{KW} and \cite{Wa4}, we have
\begin{equation}\label{a7}
D^{2}=-\sum_{i,j}g^{i,j}\Big[\partial_{i}\partial_{j}+2\sigma_{i}\partial_{j}+(\partial_{i}\sigma_{j})+\sigma_{i}\sigma_{j}
    -\Gamma_{i,j}^{k}\partial_{i}-\Gamma_{i,j}^{k}\sigma_{k}\Big]+\frac{1}{4}s,
\end{equation}
\begin{align}\label{a8}
c(df)Df^{-1}
=&\sum_{i,j}g^{i,j}\Big[c(df)c(\partial_{i})f^{-1}\Big]\partial_{j}
+\sum_{i,j}g^{i,j}\Big[c(df)c(\partial_{i})\partial_{j}\big(f^{-1}\big)+c(df)c(\partial_{i})
\sigma_{j}f^{-1}\Big],
\end{align}
where $\partial_{i}$ denotes a natural local frame on the tangent bundle $TM$, $\sigma_{i}:=-\frac{1}{4}\sum_{s,t}\omega_{s,t}(\partial_{i})e_se_t$, $g_{i,j}=g(\partial_{i},\partial_{j})$ and $(g^{i,j})_{1\leq i,j\leq n}$ is the inverse matrix associated with the metric matrix $(g_{i,j})_{1\leq i,j\leq n}$ on $M$, $\Gamma_{ kl}^{j}$ is the  Christoffel coefficient of $\nabla^{L}$.

Let
$$\nabla^{S}_{\partial_{i}}=\partial_{i}+\sigma_{i},~
\partial^{j}=g^{i,j}\partial_{i},~\sigma^{i}=g^{i,j}\sigma_{j},~\Gamma^{k}=g^{i,j}\Gamma^{k}_{i,j}.$$
Combining (\ref{a5}), (\ref{a6}), (\ref{a7}) and (\ref{a8}), we have
\begin{align}\label{az9}
\overline{D}^2
=& -g^{i,j}\partial_{i}\partial_{j}+\Big[-2\sigma^{j}+\Gamma^{j}
+c(df)c(\partial^{j})f^{-1}\Big]\partial_{j}
+g^{i,j}\Big[-(\partial_{i}\sigma_{j})-\sigma_{i}\sigma_{j}\nonumber\\
&+\Gamma_{i,j}^{k}\sigma_{k}+c(df)c(\partial_{i})\partial_{j}\big(f^{-1}\big)
+c(df)c(\partial_{i})\sigma_{j}f^{-1}\Big]+\frac{1}{4}s-\frac{1}{f^2}|df|^2 .
\end{align}

Next, let's recall the basic notions of Laplace type operators. Consider a smooth, compact, oriented Riemannian manifold $M$ of dimension $n$ without boundary, and let $V'$ be a vector bundle over $M$. Any differential operator $P$ of Laplace type locally takes the form:
\begin{equation}\label{a10}
P=-(g^{ij}\partial_i\partial_j+A^i\partial_i+B),
\end{equation}
where $A^{i}$ and $B$ are smooth sections of $\textrm{End}(V')$ (endomorphisms) over $M$.

If $P$ is a Laplace type operator with the form above, then there exists a unique connection $\nabla$ on $V'$ and a unique endomorphism $E$ such that:
 \begin{equation}\label{a11}
P=-\Big[g^{ij}(\nabla_{\partial_{i}}\nabla_{\partial_{j}}-
 \nabla_{\nabla^{L}_{\partial_{i}}\partial_{j}})+E\Big].
\end{equation}
Moreover
(with local frames of $T^{*}M$ and $V'$), $\nabla_{\partial_{i}}=\partial_{i}+\omega_{i} $
and $E$ are related to $g^{ij}$, $A^{i}$ and $B$ through
 \begin{eqnarray}\label{a12}
&&\omega_{i}=\frac{1}{2}g_{ij}\big(A^{i}+g^{kl}\Gamma_{ kl}^{j} {\rm id}\big),\nonumber\\
&&E=B-g^{ij}\big(\partial_{i}(\omega_{j})+\omega_{i}\omega_{j}-\omega_{k}\Gamma_{ ij}^{k} \big).
\end{eqnarray}

Now, let's proceed to establish the main theorem in this section.

\begin{thm}\label{th:311}
 Let $f$ be a nonzero smooth function on $M$ and $\widetilde{D}=\frac{1}{2}fD+c(df)$ be transformation operators, $\overline{D}^{2}$ as defined in (\ref{az9}), then we have
 \begin{align}\label{z6}
\overline{D}^2=&-\Big[g^{ij}(\nabla_{\partial_{i}}\nabla_{\partial_{j}}-
\nabla_{\nabla^{L}_{\partial_{i}}\partial_{j}})\Big]
+\frac{1}{4}s-\frac{1}{f^2}|df|^2-\frac{1}{2}\big[\sum^{n}_{j=1}
c(\nabla^{L}_{e_{j}}grad~f)c(e_{j})f^{-1}+c(df)c(df^{-1})\big]\nonumber\\
&+\frac{1}{4}\sum^{n}_{j=1}\big[c(df)c(e_{j})f^{-1}\big]^{2},
\end{align}
where $s$ is the scalar curvature, $\nabla_{\partial_{i}}$ is defined by (\ref{a15}) and $X=\partial_{i}$.
\end{thm}

\begin{proof}
By (\ref{az9}) and (\ref{a12}), we have
\begin{align}\label{a14}
\omega_{i}=&\sigma_i-\frac{1}{2}\big[c(df)c(\partial_i)f^{-1}\big],\nonumber\\
E=&g^{i,j}\Big[\partial_{i}(\sigma_{j})+\sigma_{i}\sigma_{j}-\Gamma_{i,j}^{k}\sigma_{k}
-c(df)c(\partial_{i})\partial_{j}\big(f^{-1}\big)
-c(df)c(\partial_{i})\sigma_{j}f^{-1}\Big]
-\frac{1}{4}s+\frac{1}{f^2}|df|^2\nonumber\\
&-g^{i,j}\bigg\{\partial_{i}
\Big[ \sigma_j-\frac{1}{2}\big[c(df)c(\partial_j)f^{-1}\big]
\Big]+\bigg[\sigma_i-\frac{1}{2}\big[c(df)c(\partial_i)f^{-1}\big]\bigg]
\times\bigg[\sigma_j-\frac{1}{2}\big[c(df)c(\partial_j)f^{-1}\big]\bigg]\nonumber\\
&-\bigg[\sigma_k-\frac{1}{2}\big[c(df)c(\partial_k)f^{-1}\big]\bigg]\Gamma^{k}_{ij}
\bigg\}.
\end{align}
For a smooth vector filed $X$ on $M$, let $c(X)$ denote the Clifford action. So,
\begin{align}\label{a15}
\nabla_{X}=\nabla^{S}_{X}-\frac{1}{2}\big[c(df)c(X)f^{-1}\big].
\end{align}
Since $E$ is globally
defined on $M$, taking normal coordinates at $x_0$, we have
$\sigma^{i}(x_0)=0$, $a^{i}(x_0)=0$, $\partial^{j}[c(\partial_{j})](x_0)=0$,
$\Gamma^k(x_0)=0$, $g^{ij}(x_0)=\delta^j_i$, so that
\begin{align}
E(x_0)
=&\bigg\{-\frac{1}{4}s+\frac{1}{f^2}|df|^2+\frac{1}{2}\big[\sum^{n}_{j=1}
c(\nabla^{L}_{e_{j}}grad~f)c(e_{j})f^{-1}+c(df)c(df^{-1})\big]\nonumber\\
&-\frac{1}{4}\sum^{n}_{j=1}\big[c(df)c(e_{j})f^{-1}\big]^{2}\bigg\}(x_0),
\end{align}
which, together with (\ref{a11}), yields Theorem \ref{th:311}.
\end{proof}

The non-commutative residue of a generalized laplacian $\widetilde{\Delta}$ is expressed as by \cite{Ac}

\begin{equation}
(n-2)\Phi_{2}(\widetilde{\Delta})=(4\pi)^{-\frac{n}{2}}\Gamma(\frac{n}{2})\widetilde{res}(\widetilde{\Delta}^{-\frac{n}{2}+1}),
\end{equation}
where $\Phi_{2}(\widetilde{\Delta})$ denotes the integral over the diagonal part of the second
coefficient of the heat kernel expansion of $\widetilde{\Delta}$.
Now let $\widetilde{\Delta}=\overline{D}^2$. By \cite{Y} and \cite{Ac}, we have
\begin{equation}
\Delta=-g^{ij}(\nabla^L_{i}\nabla^L_{j}-\Gamma_{ij}^{k}\nabla^L_{k}).
\end{equation}
Let's assume that $n=4$. Since $\overline{D}^2$ is a generalized laplacian, we can suppose $\overline{D}^2=\Delta-E$, then we have
\begin{align}
{\rm Wres}(\overline{D}^{-2})=&\frac{(4-2)(4\pi)^{\frac{4}{2}}}{(\frac{4}{2}-1)!}\int_{M}{\rm trace}(\frac{1}{6}s+E_{\overline{D}^2})d{\rm vol_{M}}\nonumber\\
=&32\pi^2
\int_{M}{\rm trace}\bigg\{-\frac{1}{12}s+\frac{1}{f^2}|df|^2+\frac{1}{2}\big[\sum^{4}_{j=1}
c(\nabla^{L}_{e_{j}}grad~f)c(e_{j})f^{-1}+c(df)c(df^{-1})\big]\nonumber\\
&-\frac{1}{4}\sum^{4}_{j=1}\big[c(df)c(e_{j})f^{-1}\big]^{2}\bigg\}d{\rm vol_{M}}\nonumber\\
=&-128\pi^{2}
\int_{M}\bigg\{\frac{1}{12}s+\frac{\Delta f}{2f}
+\frac{1}{2}g(df,df^{-1})+\frac{2|df|^{2}}{f^{2}}\bigg\}d{\rm vol_{M} },
\end{align}
where ${\rm Wres}$ denote the noncommutative residue. Therefore, we obtain
\begin{align}
{\rm Wres}(\widetilde{D}^{-2})
={\rm Wres}(\frac{1}{4}\overline{D}^2f^2)^{-1}
={\rm Wres}(4f^{-2}\overline{D}^{-2}).
\end{align}

\begin{thm} \label{th:312}
For $4$-dimensional compact oriented manifolds without boundary, the following equalities holds:
\begin{align}
{\rm Wres}(\widetilde{D}^2)^{-\frac{n-2}{2}}
=&
\int_{M}\frac{-512\pi^{2}}{f^{2}}\bigg\{\frac{1}{12}s+\frac{\Delta f}{2f}
+\frac{1}{2}g(df,df^{-1})+\frac{2|df|^{2}}{f^{2}}\bigg\}d{\rm vol_{M} },
\end{align}
where $s$ is the scalar curvature.
\end{thm}

\section{A Kastler-Kalau-Walze Type Theorem for Perturbations of
Dirac Operators on Manifolds with Boundary}
Let $M$ be a $4$-dimensional compact spin manifolds with the boundary $\partial M$.
In the following, we will compute the case $\widetilde{{\rm Wres}}[\pi^{+}\widetilde{D}^{-1} \circ\pi^{+}\widetilde{D}^{-1}]$.
An application of (\ref{a16}) and (\ref{a17}) shows that
\begin{eqnarray}\label{b1}
\widetilde{{\rm Wres}}[\pi^{+}\widetilde{D}^{-1} \circ\pi^{+}\widetilde{D}^{-1}]=\int_{M}\int_{|\xi|=1}{{\rm trace}}_{S(TM)}\big[\sigma_{-4}\big( \widetilde{D}^{-2}\big)\big]\sigma(\xi)dx+\int_{\partial M}\Phi,
\end{eqnarray}
where
 \begin{eqnarray}\label{b2}
\Phi&=&\int_{|\xi'|=1}\int_{-\infty}^{+\infty}\sum_{j,k=0}^{\infty}\sum \frac{(-i)^{|\alpha|+j+k+\ell}}{\alpha!(j+k+1)!}
{{\rm trace}}_{S(TM)}\Big[\partial_{x_{n}}^{j}\partial_{\xi'}^{\alpha}\partial_{\xi_{n}}^{k}\sigma_{r}^{+}
(\widetilde{D}^{-1})(x',0,\xi',\xi_{n})\nonumber\\
&&\times\partial_{x'}^{\alpha}\partial_{\xi_{n}}^{j+1}\partial_{x_{n}}^{k}\sigma_{l}
\Big(\widetilde{D}^{-1}\Big)(x',0,\xi',\xi_{n})\Big]
d\xi_{n}\sigma(\xi')dx' ,
\end{eqnarray}
the sum is taken over $r-k+|\alpha|+\ell-j-1=-n=-4,r\leq-1,\ell\leq-1$
and $\sigma^+_r(\widetilde{D}^{-1})=\pi^{+}_{\xi_n}\sigma^+_r(\widetilde{D}^{-1})$.

Locally we can use Theorem \ref{th:312} to compute the interior of $\widetilde{{\rm Wres}}[\pi^{+}\widetilde{D}^{-1} \circ\pi^{+}\widetilde{D}^{-1}]$, we have
\begin{align}\label{a19}
&\int_{M}\int_{|\xi|=1}{{\rm trace}}_{S(TM)}\big[\sigma_{-4}\big( \widetilde{D}^{-2}\big)\big]\sigma(\xi)dx\nonumber\\
=&
\int_{M}\frac{-512\pi^{2}}{f^{2}}\bigg\{\frac{1}{12}s+\frac{\Delta f}{2f}
+\frac{1}{2}g(df,df^{-1})+\frac{2|df|^{2}}{f^{2}}\bigg\}d{\rm vol_{M} }.
\end{align}
So we only need to compute $\int_{\partial M}\Phi$.

By Lemma 1 in \cite{Wa4} and Lemma 2.1 in \cite{Wa3}, for any fixed point $x_0\in\partial M$, choosing the normal coordinates $U$
 of $x_0$ in $\partial M$ (not in $M$). Denote by $\sigma_{l}(P)$ the $l$-order symbol of an operator $P$. By the composition formula and (2.2.11) in \cite{Wa3}, we can prove the following useful result.

\begin{lem}\label{le:4.1}
Let $D$ be the Dirac operators on $\Gamma(S(TM))$ and $\widetilde{D}$ be the perturbations of Dirac operators on  $\Gamma(S(TM))$, the symbolic calculus of pseudo-differential operators yields
\begin{eqnarray}\label{z:1.1}
\sigma_{-1}(D^{-1})&=&\frac{ic(\xi)}{|\xi|^2}; \nonumber\\
\sigma_{-1}(\widetilde{D})^{-1}&=&2f^{-1}\sigma_{-1}(D^{-1})=\frac{2ic(\xi)}{f|\xi|^2}; \nonumber\\
\sigma_{-2}(D)^{-1}&=&\frac{c(\xi)
\sigma_{0}(D)c(\xi)}{|\xi|^4}+\frac{c(\xi)}{|\xi|^6}\sum_jc(dx_j)
\Big[\partial_{x_j}[c(\xi)]|\xi|^2-c(\xi)\partial_{x_j}(|\xi|^2)\Big];\nonumber\\
\sigma_{-2}(\widetilde{D})^{-1}&=&\frac{4c(\xi)\sigma_{0}(\widetilde{D})c(\xi)}{f^{2}|\xi|^4}
+\frac{c(\xi)}{|\xi|^6}\sum_jc(dx_j)\Big[2\partial_{x_j}(f^{-1})c(\xi)|\xi|^2+2f^{-1} \partial_{x_j}[c(\xi)]|\xi|^2\nonumber\\
&&-2f^{-1}c(\xi)\partial_{x_j}(|\xi|^2)\Big]\nonumber\\
&=&\frac{2}{f}\sigma_{-2}(D)^{-1}+\frac{4}{f^{2}}
\cdot\frac{c(\xi)c(df)c(\xi)}{|\xi|^4}
+\frac{c(\xi)\sum_jc(dx_j)\cdot2\partial_{x_j}(f^{-1})c(\xi)}{|\xi|^4};
\end{eqnarray}
where
\begin{align}
\sigma_0(D)&=-\frac{1}{4}\sum_{l,s,t}\omega_{s,t}(e_l)c(e_{l})c(e_s)c(e_t);\nonumber\\
\sigma_0(\widetilde{D})&=\bigg[-\frac{1}{4}\sum_{l,s,t}\omega_{s,t}(e_l)c(e_{l})c(e_s)c(e_t)
\bigg]\frac{f}{2}+c(df)=\frac{1}{2}\sigma_0(D)f+c(df).
\end{align}
\end{lem}

From the formula (\ref{b2}) for the definition of $\Phi$, now we can compute $\Phi$.
Since the sum is taken over $r+\ell-k-j-|\alpha|-1=-4, \ r\leq-1, \ell\leq -1$, then we have the $\int_{\partial{M}}\Phi$
is the sum of the following five cases:
~\\
~\\
\noindent  {\bf case (a)~(I)}~$r=-1, l=-1, j=k=0, |\alpha|=1$.
\\
~\\
By (\ref{b2}), we get
 \begin{eqnarray}
&&{\rm case~(a)~(I)}\nonumber\\
&=&-\int_{|\xi'|=1}\int^{+\infty}_{-\infty}\sum_{|\alpha|=1}{\rm trace}
\Big[\partial^{\alpha}_{\xi'}\pi^{+}_{\xi_{n}}\sigma_{-1}(\widetilde{D}^{-1})
      \times\partial^{\alpha}_{x'}\partial_{\xi_{n}}\sigma_{-1}
      \big(\widetilde{D}^{-1}\big)\Big](x_0)d\xi_n\sigma(\xi')dx'\nonumber\\
      &=&-\int_{|\xi'|=1}\int^{+\infty}_{-\infty}\sum_{|\alpha|=1}{\rm trace}
\Big[\partial^{\alpha}_{\xi'}\pi^{+}_{\xi_{n}}\big(2f^{-1}\sigma_{-1}(D^{-1})\big)
      \times\partial^{\alpha}_{x'}\partial_{\xi_{n}}\big(2f^{-1}\sigma_{-1}(D^{-1})\big)\Big](x_0)
      d\xi_n\sigma(\xi')dx'\nonumber\\
      &=&-2f^{-1}\times 2f^{-1}\int_{|\xi'|=1}\int^{+\infty}_{-\infty}\sum_{|\alpha|=1}{\rm trace}
\Big[\partial^{\alpha}_{\xi'}\pi^{+}_{\xi_{n}}\sigma_{-1}(D^{-1})
      \times\partial^{\alpha}_{x'}\partial_{\xi_{n}}\sigma_{-1}(D^{-1})\Big]
      (x_0)d\xi_n\sigma(\xi')dx'\nonumber\\
      &&-2f^{-1}\sum\limits_{j<n}\partial_{j}(2f^{-1})
      \int_{|\xi'|=1}\int^{+\infty}_{-\infty}\sum_{|\alpha|=1}{\rm trace}
\Big[\partial^{\alpha}_{\xi'}\pi^{+}_{\xi_{n}}\sigma_{-1}(D^{-1})
      \times\partial_{\xi_{n}}\sigma_{-1}(D^{-1})\Big](x_0)\nonumber\\
      &&\times d\xi_n\sigma(\xi')dx'.
\end{eqnarray}
By Lemma \ref{le:3112}, for $i<n$, we have
 \begin{align}\label{b3}
\partial_{x_i}\left(\frac{ic(\xi)}{|\xi|^2}\right)(x_0)=
\frac{i\partial_{x_i}[c(\xi)](x_0)}{|\xi|^2}
-\frac{ic(\xi)\partial_{x_i}(|\xi|^2)(x_0)}{|\xi|^4}=0.
\end{align}
Thus we have
 \begin{align}\label{b3}
-2f^{-1}\times 2f^{-1}\int_{|\xi'|=1}\int^{+\infty}_{-\infty}\sum_{|\alpha|=1}{\rm trace}
\Big[\partial^{\alpha}_{\xi'}\pi^{+}_{\xi_{n}}\sigma_{-1}(D^{-1})
      \times\partial^{\alpha}_{x'}\partial_{\xi_{n}}\sigma_{-1}(D^{-1})\Big]
      (x_0)d\xi_n\sigma(\xi')dx'=0.
\end{align}

By (\ref{z:1.1}) and direct calculations, for $i<n$, we obtain
\begin{eqnarray}
&&\partial^{\alpha}_{\xi'}\pi^{+}_{\xi_{n}}\sigma_{-1}(D^{-1})(x_0)|_{|\xi'|=1}
=\partial_{\xi_i}\pi^{+}_{\xi_{n}}\sigma_{-1}(D^{-1})(x_0)|_{|\xi'|=1}\nonumber\\
&=&\frac{c(dx_i)}{2(\xi_n-\sqrt{-1})}-\frac{\xi_i(\xi_n-2\sqrt{-1})c(\xi')
+\xi_ic(dx_n)}{2(\xi_n-\sqrt{-1})^2}
\end{eqnarray}
and
\begin{align}\label{z:1.2}
\partial_{\xi_n}\sigma_{-1}(D)^{-1}|_{|\xi'|=1}(x_0)
=i\bigg[\frac{1-\xi^2_n}{(1+\xi^2_n)^2}c(dx_n)
-\frac{2\xi_n}{(1+\xi^2_n)^2}c(\xi')\bigg].
\end{align}

Then for $i<n$, we have
\begin{eqnarray}
&&{\rm trace}
\Big[\partial^{\alpha}_{\xi'}\pi^{+}_{\xi_{n}}\sigma_{-1}(D^{-1})
      \times\partial_{\xi_{n}}\sigma_{-1}(D^{-1})\Big](x_0)\nonumber\\
&=&-\frac{2i\xi_i\xi_n}{2(\xi_n-i)(1+\xi^2_n)^2}{\rm trace}[c(dx_i)^2]
+\frac{2i\xi_n\xi_i(\xi_n-2i)}{2(\xi_n-i)^2(1+\xi^2_n)^2}{\rm trace}[c(\xi')^2]\nonumber\\
&&-\frac{i\xi_i(1-\xi^2_n)}{2(\xi_n-i)^2(1+\xi^2_n)^2}{\rm trace}[c(dx_n)^2].
\end{eqnarray}
We note that $i<n,~\int_{|\xi'|=1}\xi_i\sigma(\xi')=0$,
so
\begin{eqnarray}
&&-2f^{-1}\sum\limits_{j<n}\partial_{j}(2f^{-1})
      \int_{|\xi'|=1}\int^{+\infty}_{-\infty}\sum_{|\alpha|=1}{\rm trace}
\Big[\partial^{\alpha}_{\xi'}\pi^{+}_{\xi_{n}}\sigma_{-1}(D^{-1})
      \times\partial_{\xi_{n}}\sigma_{-1}(D^{-1})\Big](x_0) d\xi_n\sigma(\xi')dx'\nonumber\\
&=&0.
\end{eqnarray}
Then we have ${\bf case~(a)~(I)}=0$.
~\\

\noindent  {\bf case (a)~(II)}~$r=-1,~l=-1,~k=|\alpha|=0,~j=1$.\\

\noindent By (\ref{b2}), we get
\begin{align}\label{b4}
&{\rm case~(a)~(II)}\nonumber\\
=&-\frac{1}{2}\int_{|\xi'|=1}\int^{+\infty}_{-\infty} {\rm
trace} [\partial_{x_n}\pi^+_{\xi_n}\sigma_{-1}(\widetilde{D}^{-1})\times
\partial_{\xi_n}^2\sigma_{-1}(\widetilde{D}^{-1})](x_0)d\xi_n\sigma(\xi')dx'\nonumber\\
=&-\frac{1}{2}\int_{|\xi'|=1}\int^{+\infty}_{-\infty} {\rm
trace} [\partial_{x_n}\pi^+_{\xi_n}\big(2f^{-1}\sigma_{-1}(D^{-1})\big)\times
\partial_{\xi_n}^2\big(2f^{-1}\sigma_{-1}(D^{-1})\big)](x_0)d\xi_n\sigma(\xi')dx'\nonumber\\
=&-\frac{1}{2}\times2f^{-1}\times2f^{-1}
\int_{|\xi'|=1}\int^{+\infty}_{-\infty}{\rm trace}
\Big[\partial_{x_{n}}\pi^{+}_{\xi_{n}}\sigma_{-1}(D^{-1})
      \times\partial^2_{\xi_{n}}\sigma_{-1}(D^{-1})\Big](x_0)d\xi_n\sigma(\xi')dx'\nonumber\\
      &-\frac{1}{2}\times 2f^{-1}\partial_{x_{n}}(2f^{-1})
      \int_{|\xi'|=1}\int^{+\infty}_{-\infty}{\rm trace}
\Big[\pi^{+}_{\xi_{n}}\sigma_{-1}(D^{-1})
      \times\partial^2_{\xi_{n}}\sigma_{-1}(D^{-1})\Big](x_0)d\xi_n\sigma(\xi')dx'.
\end{align}
\noindent By Lemma \ref{le:4.1}, we have
\begin{align}\label{b5}
\partial^2_{\xi_n}\sigma_{-1}(D^{-1})(x_0)
=i\left(-\frac{6\xi_nc(dx_n)+2c(\xi')}
{|\xi|^4}+\frac{8\xi_n^2c(\xi)}{|\xi|^6}\right);
\end{align}
\begin{align}\label{b6}
\partial_{x_n}\sigma_{-1}(D^{-1})(x_0)
=\frac{i\partial_{x_n}c(\xi')(x_0)}{|\xi|^2}
-\frac{ic(\xi)|\xi'|^2h'(0)}{|\xi|^4}.
\end{align}
By (2.2.23) in \cite{Wa3} and (3.12), we have
\begin{equation}
\pi^{+}_{\xi_{n}}\partial_{x_{n}}\sigma_{-1}(D^{-1})(x_{0})|_{|\xi'|=1}
=\frac{\partial_{x_{n}}[c(\xi')](x_{0})}{2(\xi_{n}-\sqrt{-1})}
+\sqrt{-1}h'(0)\bigg[\frac{\sqrt{-1}c(\xi')}{4(\xi_{n}-\sqrt{-1})}
+\frac{c(\xi')+\sqrt{-1}c(dx_{n})}{4(\xi_{n}-\sqrt{-1})^{2}}\bigg].
\end{equation}

\noindent By the relation of the Clifford action and ${\rm trace}{AB}={\rm trace}{BA}$, we have the equalities:
\begin{align}\label{32}
&{\rm trace}[c(\xi')c(dx_n)]=0;~~{\rm trace}[c(dx_n)^2]=-4;~~{\rm trace}[c(\xi')^2](x_0)|_{|\xi'|=1}=-4;\nonumber\\
&{\rm trace}[\partial_{x_n}c(\xi')c(dx_n)]=0;~~{\rm trace}[\partial_{x_n}c(\xi')c(\xi')](x_0)|_{|\xi'|=1}=-2h'(0).
\end{align}

Then
\begin{align}\label{35}
&-\frac{1}{2}\times2f^{-1}\times2f^{-1}
\int_{|\xi'|=1}\int^{+\infty}_{-\infty}{\rm trace}
\Big[\partial_{x_{n}}\pi^{+}_{\xi_{n}}\sigma_{-1}(D^{-1})
      \times\partial^2_{\xi_{n}}\sigma_{-1}(D^{-1})\Big](x_0)d\xi_n\sigma(\xi')dx'\nonumber\\
=&-\int_{|\xi'|=1}\int^{+\infty}_{-\infty}\frac{4ih'(0)(\xi_n-i)^2}
{f^2(\xi_n-i)^4(\xi_n+i)^3}d\xi_n\sigma(\xi')dx'\nonumber\\
=&-\frac{4 ih'(0)\Omega_3}{f^2}\int_{\Gamma^+}\frac{1}{(\xi_n-i)^2(\xi_n+i)^3}d\xi_ndx'\nonumber\\
=&-\frac{4 ih'(0)\Omega_32\pi i}{f^2}\bigg[\frac{1}{(\xi_n+i)^3}\bigg]^{(1)}\bigg|_{\xi_n=i}dx'\nonumber\\
=&-\frac{3}{8}\times\frac{4}{f^2}\pi h'(0)\Omega_3dx'\nonumber\\
=&-\frac{3}{2f^2}\pi h'(0)\Omega_3dx',
\end{align}
where ${\rm \Omega_{3}}$ is the canonical volume of 3-dimensional unit
sphere $S^{2}$.\\

On the other hand, by calculations, we have
\begin{equation}
\pi^{+}_{\xi_{n}}\sigma_{-1}(D^{-1})(x_{0})|_{|\xi'|=1}
=-\frac{c(\xi')+\sqrt{-1}c(dx_{n})}{2(\xi_{n}-\sqrt{-1})},
\end{equation}
then
\begin{align}\label{35}
&-\frac{1}{2}\times 2f^{-1}\partial_{x_{n}}(2f^{-1})
      \int_{|\xi'|=1}\int^{+\infty}_{-\infty}{\rm trace}
\Big[\pi^{+}_{\xi_{n}}\sigma_{-1}(D^{-1})
      \times\partial^2_{\xi_{n}}\sigma_{-1}(D^{-1})\Big](x_0)d\xi_n\sigma(\xi')dx'\nonumber\\
=&-\frac{1}{2}\times 2f^{-1}\partial_{x_{n}}(2f^{-1})
      \int_{|\xi'|=1}\int^{+\infty}_{-\infty}\frac{4i(3\xi^2_n-1)+4(3\xi_n-\xi^3_n)}
{(\xi_n-i)^4(\xi_n+i)^3}d\xi_n\sigma(\xi')dx'\nonumber\\
=&-\frac{1}{2}\times 2f^{-1}\partial_{x_{n}}(2f^{-1})\Omega_3\int_{\Gamma^+}
\frac{4i(3\xi^2_n-1)+4(3\xi_n-\xi^3_n)}
{(\xi_n-i)^4(\xi_n+i)^3}d\xi_ndx'\nonumber\\
=&-\frac{1}{2}\times 2f^{-1}\partial_{x_{n}}(2f^{-1})\Omega_3
\bigg[\frac{4i(3\xi^2_n-1)+4(3\xi_n-\xi^3_n)}
{(\xi_n+i)^3}\bigg]^{(3)}\bigg|_{\xi_n=i}dx'\nonumber\\
=&0,
\end{align}

Then we have ${\bf case~(a)~(II)}=-\frac{3}{2f^2}\pi h'(0)\Omega_3dx'$.
~\\

\noindent  {\bf case (a)~(III)}~$r=-1,~l=-1,~j=|\alpha|=0,~k=1$.\\

\noindent By (\ref{b2}), we get
\begin{align}\label{36}
&{\bf case (a)~(III)}\nonumber\\
=&-\frac{1}{2}\int_{|\xi'|=1}\int^{+\infty}_{-\infty}
{\rm trace} [\partial_{\xi_n}\pi^+_{\xi_n}\sigma_{-1}(\widetilde{D})^{-1}\times
\partial_{\xi_n}\partial_{x_n}\sigma_{-1}(\widetilde{D})^{-1}]
(x_0)d\xi_n\sigma(\xi')dx'\nonumber\\
=&-\frac{1}{2}\int_{|\xi'|=1}\int^{+\infty}_{-\infty}
{\rm trace} \big[\partial_{\xi_n}\pi^+_{\xi_n}\big(2f^{-1}\sigma_{-1}(D^{-1})\big)\times
\partial_{\xi_n}\partial_{x_n}\big(2f^{-1}\sigma_{-1}(D^{-1})\big)\big]
(x_0)d\xi_n\sigma(\xi')dx'\nonumber\\
=&-\frac{1}{2}\times 2f^{-1}\times 2f^{-1}
 \int_{|\xi'|=1}\int^{+\infty}_{-\infty}{\rm trace}
\Big[\partial_{\xi_{n}}\pi^{+}_{\xi_{n}}\sigma_{-1}(D^{-1})
      \times\partial_{\xi_{n}}\partial_{x_{n}}\sigma_{-1}
(D^{-1})\Big](x_0)d\xi_n\sigma(\xi')dx'\nonumber\\
      &-\frac{1}{2}\times 2f^{-1}\partial_{x_{n}}(2f^{-1})
      \int_{|\xi'|=1}\int^{+\infty}_{-\infty}{\rm trace}
\Big[\partial_{\xi_{n}}\pi^{+}_{\xi_{n}}\sigma_{-1}(D^{-1})
      \times\partial_{\xi_{n}}\sigma_{-1}
(D^{-1})\Big](x_0) d\xi_n\sigma(\xi')dx'.
\end{align}
By Lemma \ref{le:4.1}, we have
\begin{align}\label{37}
\partial_{\xi_n}\partial_{x_n}\sigma_{-1}(D^{-1})(x_0)|_{|\xi'|=1}
=-ih'(0)\left[\frac{c(dx_n)}{|\xi|^4}-4\xi_n\frac{c(\xi')
+\xi_nc(dx_n)}{|\xi|^6}\right]-\frac{2\xi_ni\partial_{x_n}c(\xi')(x_0)}{|\xi|^4};
\end{align}
\begin{align}\label{38}
\partial_{\xi_n}\pi^+_{\xi_n}\sigma_{-1}(D^{-1})(x_0)|_{|\xi'|=1}
=-\frac{c(\xi')+ic(dx_n)}{2(\xi_n-i)^2}.
\end{align}
By (\ref{32}), we have
\begin{align}\label{39}
{\rm trace}\left\{\frac{c(\xi')+ic(dx_n)}{2(\xi_n-i)^2}\times
ih'(0)\left[\frac{c(dx_n)}{|\xi|^4}-4\xi_n\frac{c(\xi')+\xi_nc(dx_n)}{|\xi|^6}\right]\right\}
=\frac{2h'(0)(i-3\xi_n)}{(\xi_n-i)^4(\xi_n+i)^3}
\end{align}
and
\begin{align}\label{40}
{\rm tr}\left[\frac{c(\xi')+ic(dx_n)}{2(\xi_n-i)^2}\times
\frac{2\xi_ni\partial_{x_n}c(\xi')(x_0)}{|\xi|^4}\right]
=\frac{-2ih'(0)\xi_n}{(\xi_n-i)^4(\xi_n+i)^2}.
\end{align}
So we have
\begin{align}\label{41}
&-\frac{1}{2}\times 2f^{-1}\times 2f^{-1}
 \int_{|\xi'|=1}\int^{+\infty}_{-\infty}{\rm trace}
\Big[\partial_{\xi_{n}}\pi^{+}_{\xi_{n}}\sigma_{-1}(D^{-1})
      \times\partial_{\xi_{n}}\partial_{x_{n}}\sigma_{-1}
(D^{-1})\Big](x_0)d\xi_n\sigma(\xi')dx'\nonumber\\
=&-\int_{|\xi'|=1}\int^{+\infty}_{-\infty}\frac{4h'(0)(i-3\xi_n)}
{f^2(\xi_n-i)^4(\xi_n+i)^3}d\xi_n\sigma(\xi')dx'
-\int_{|\xi'|=1}\int^{+\infty}_{-\infty}\frac{4h'(0)i\xi_n}
{f^2(\xi_n-i)^4(\xi_n+i)^2}d\xi_n\sigma(\xi')dx'\nonumber\\
=&-h'(0)\Omega_3\frac{8\pi i}{3!f^2}\left[\frac{(i-3\xi_n)}{(\xi_n+i)^3}\right]^{(3)}
\bigg|_{\xi_n=i}dx'+h'(0)\Omega_3\frac{8\pi i}{3!f^2}\left[\frac{i\xi_n}{(\xi_n+i)^2}\right]^{(3)}\bigg|_{\xi_n=i}dx'\nonumber\\
=&\frac{4}{f^2}\times\frac{3}{8}\pi h'(0)\Omega_3dx'\nonumber\\
=&\frac{3}{2f^2}\pi h'(0)\Omega_3dx'.
\end{align}

On the other hand, we have
\begin{align}\label{z:1.2}
\partial_{\xi_n}\sigma_{-1}(D)^{-1}|_{|\xi'|=1}(x_0)
=i\bigg[\frac{1-\xi^2_n}{(1+\xi^2_n)^2}c(dx_n)
-\frac{2\xi_n}{(1+\xi^2_n)^2}c(\xi')\bigg],
\end{align}
then we have
\begin{align}\label{za:1.2}
&-\frac{1}{2}\times 2f^{-1}\partial_{x_{n}}(2f^{-1})
      \int_{|\xi'|=1}\int^{+\infty}_{-\infty}{\rm trace}
\Big[\partial_{\xi_{n}}\pi^{+}_{\xi_{n}}\sigma_{-1}(D^{-1})
      \times\partial_{\xi_{n}}\sigma_{-1}
(D^{-1})\Big](x_0) d\xi_n\sigma(\xi')dx'\nonumber\\
=&-\frac{1}{2}\times 2f^{-1}\partial_{x_{n}}(2f^{-1})
      \int_{|\xi'|=1}\int^{+\infty}_{-\infty}
\Big[\frac{2(\xi^2_n-1)-2\xi_n}{(\xi_n-i)^4(\xi_n+i)^2}\Big]
(x_0) d\xi_n\sigma(\xi')dx'\nonumber\\
=&-\frac{1}{2}\times 2f^{-1}\partial_{x_{n}}(2f^{-1})\Omega_3\int_{\Gamma^+}
\Big[\frac{2(\xi^2_n-1)-2\xi_n}{(\xi_n-i)^4(\xi_n+i)^2}\Big]d\xi_ndx'\nonumber\\
=&-\frac{1}{2}\times 2f^{-1}\partial_{x_{n}}(2f^{-1})\Omega_3
\bigg[\frac{2(\xi^2_n-1)-2\xi_n}{(\xi_n+i)^2}\bigg]^{(3)}\bigg|_{\xi_n=i}dx'\nonumber\\
=&-\frac{1}{2}\times 2f^{-1}\partial_{x_{n}}(2f^{-1})\Omega_3\times \big(\frac{3}{2}-\frac{9i}{4}\big)dx'\nonumber\\
=&\frac{9i-6}{2} f^{-1}\partial_{x_{n}}(f^{-1})\Omega_3dx'.
\end{align}

Then we have ${\bf case~(a)~(III)}=\frac{3}{2f^2}\pi h'(0)\Omega_3dx'+\frac{9i-6}{2} f^{-1}\partial_{x_{n}}(f^{-1})\Omega_3dx'$.
~\\

\noindent  {\bf case (b)}~$r=-2,~l=-1,~k=j=|\alpha|=0$.\\

\noindent By (\ref{b2}), we get
\begin{align}\label{42}
{\bf case~(b)}
=&-i\int_{|\xi'|=1}\int^{+\infty}_{-\infty}{\rm trace} [\pi^+_{\xi_n}\sigma_{-2}(\widetilde{D})^{-1}\times
\partial_{\xi_n}\sigma_{-1}(\widetilde{D})^{-1}](x_0)d\xi_n\sigma(\xi')dx'\nonumber\\
=&-i\int_{|\xi'|=1}\int^{+\infty}_{-\infty}{\rm tr} \bigg[\pi^+_{\xi_n}\big[\frac{2}{f}\sigma_{-2}(D^{-1})\big]\times
\partial_{\xi_n}\sigma_{-1}(\widetilde{D})^{-1}\bigg](x_0)d\xi_n\sigma(\xi')dx'\nonumber\\
&-i\int_{|\xi'|=1}\int^{+\infty}_{-\infty}{\rm trace} \bigg[\pi^+_{\xi_n}\bigg(\frac{4}{f^{2}}\cdot\frac{c(\xi)c(df)c(\xi)}{|\xi|^4}\bigg)\times
\partial_{\xi_n}\sigma_{-1}(\widetilde{D})^{-1}\bigg](x_0)d\xi_n\sigma(\xi')dx'\nonumber\\
&-i\int_{|\xi'|=1}\int^{+\infty}_{-\infty}{\rm trace} \bigg[\pi^+_{\xi_n}\bigg(\frac{c(\xi)\sum_jc(dx_j)\cdot
2\partial_{x_j}(f^{-1})c(\xi)}{|\xi|^4}\bigg)\times
\partial_{\xi_n}\sigma_{-1}(\widetilde{D})^{-1}\bigg](x_0)\nonumber\\
&\times d\xi_n\sigma(\xi')dx'.
\end{align}
By {\bf case b} in \cite{Wa3}, we have
\begin{align}\label{90}
&-i\int_{|\xi'|=1}\int^{+\infty}_{-\infty}{\rm trace} \bigg[\pi^+_{\xi_n}\big[\frac{2}{f}\sigma_{-2}(D^{-1})\big]\times
\partial_{\xi_n}\sigma_{-1}(\widetilde{D})^{-1}\bigg](x_0)d\xi_n\sigma(\xi')dx'\nonumber\\
=&\frac{4}{f^{2}}\times\frac{9}{8}\pi h'(0)\Omega_3dx'\nonumber\\
=&\frac{9}{2f^{2}}\pi h'(0)\Omega_3dx'.
\end{align}
Moreover,
\begin{align}\label{91}
&\pi^+_{\xi_n}\bigg(\frac{c(\xi)c(df)c(\xi)}{|\xi|^4}\bigg)(x_0)|_{|\xi'|=1}\nonumber\\
=&\pi^+_{\xi_n}\bigg[\frac{[c(\xi')+\xi_nc(dx_n)]c(df)[c(\xi')+\xi_nc(dx_n)]}{(1+\xi^2_n)^2}\bigg]\nonumber\\
=&\frac{1}{2\pi i}\int_{\Gamma^+}\bigg[\frac{c(\xi')c(df)c(\xi')+c(dx_n)c(df)c(\xi')\eta_n
+c(\xi')c(df)c(dx_n)\eta_n+c(dx_n)c(df)c(dx_n)\eta^{2}_n}{(\eta_n+i)^{2}(\xi_n-\eta_n)(\eta_n-i)^{2}}\bigg]d\eta_n\nonumber\\
=&\bigg[\frac{c(\xi')c(df)c(\xi')+c(dx_n)c(df)c(\xi')\eta_n
+c(\xi')c(df)c(dx_n)\eta_n+c(dx_n)c(df)c(dx_n)\eta^{2}_n}{(\eta_n+i)^{2}
(\xi_n-\eta_n)}\bigg]'|_{\eta_n=i}\nonumber\\
=&-\frac{i\xi_n+2}{4(\xi_n-i)^2}c(\xi')c(df)c(\xi')
-\frac{i}{4(\xi_n-i)^2}\bigg[c(dx_n)c(df)c(\xi')+c(\xi')c(df)c(dx_n)\bigg]\nonumber\\
&-\frac{i\xi_n}{4(\xi_n-i)^2}c(dx_n)c(df)c(dx_n)
\end{align}
and
\begin{align}\label{92}
\partial_{\xi_n}\sigma_{-1}(\widetilde{D})^{-1}|_{|\xi'|=1}(x_0)
=\frac{2i}{f}\bigg[\frac{1-\xi^2_n}{(1+\xi^2_n)^2}c(dx_n)
-\frac{2\xi_n}{(1+\xi^2_n)^2}c(\xi')\bigg].
\end{align}
By (\ref{92}) and
$$c(\xi')^2|_{|\xi'|=1}=-1;~c(dx_n)^2=-1;~c(\xi')c(dx_n)=-c(dx_n)c(\xi');~{\rm trace}~AB={\rm trace}~BA,$$
we get
\begin{align}\label{93}
&{\rm trace} \bigg[\pi^+_{\xi_n}\bigg(\frac{c(\xi)c(df)c(\xi)}{|\xi|^4}\bigg)\times
\partial_{\xi_n}\sigma_{-1}(\widetilde{D})^{-1}\bigg](x_0)|_{|\xi'|=1}\nonumber\\
=&\frac{i}{f(1+\xi^2_n)^2}{\rm trace} [c(dx_n)c(df)]+\frac{1}{f(1+\xi^2_n)^2}{\rm trace} [c(\xi')c(df)].
\end{align}
Considering, for $i<n$,
\begin{align}\label{94}
&-i\int_{|\xi'|=1}\int^{+\infty}_{-\infty}{\rm trace}
\bigg[\pi^+_{\xi_n}\bigg(\frac{4}{f^{2}}\cdot\frac{c(\xi)c(df)c(\xi)}{|\xi|^4}\bigg)\times
\partial_{\xi_n}\sigma_{-1}(\widetilde{D})^{-1}\bigg](x_0)|_{|\xi'|=1}d\xi_n\sigma(\xi')dx'\nonumber\\
=&\frac{4}{f^{2}}\cdot\frac{\pi}{4}\times\frac{2}{f}{\rm trace}[c(dx_n)c(df)]\Omega_3dx'\nonumber\\
=&\frac{2\pi}{f^{3}}{\rm trace}[c(dx_n)c(df)]\Omega_3dx'.
\end{align}
On the other hand, we have
\begin{align}\label{91}
&\pi^+_{\xi_n}\bigg(\frac{c(\xi)\sum_jc(dx_j)\cdot
2\partial_{x_j}(f^{-1})c(\xi)}{|\xi|^4}\bigg)(x_0)|_{|\xi'|=1}\nonumber\\
=&\pi^+_{\xi_n}\bigg[\frac{[c(\xi')+\xi_nc(dx_n)]\sum_jc(dx_j)\cdot
2\partial_{x_j}(f^{-1})[c(\xi')+\xi_nc(dx_n)]}{(1+\xi^2_n)^2}\bigg]\nonumber\\
=&\frac{1}{2\pi i}\int_{\Gamma^+}\Bigg\{
\bigg[
c(\xi')\sum_jc(dx_j)\cdot
2\partial_{x_j}(f^{-1})c(\xi')+c(dx_n)\sum_jc(dx_j)\cdot
2\partial_{x_j}(f^{-1})c(\xi')\eta_n\nonumber\\
&+c(\xi')\sum_jc(dx_j)\cdot
2\partial_{x_j}(f^{-1})c(dx_n)\eta_n+c(dx_n)\sum_jc(dx_j)\cdot
2\partial_{x_j}(f^{-1})c(dx_n)\eta^{2}_n\bigg]\nonumber\\
&\times\bigg[(\eta_n+i)^{-2}
(\xi_n-\eta_n)^{-1}(\eta_n-i)^{-2}\bigg]\Bigg\}
d\eta_n\nonumber\\
=&\Bigg\{\bigg[c(\xi')\sum_jc(dx_j)\cdot
2\partial_{x_j}(f^{-1})c(\xi')+c(dx_n)\sum_jc(dx_j)\cdot
2\partial_{x_j}(f^{-1})c(\xi')\eta_n\nonumber\\
&+c(\xi')\sum_jc(dx_j)\cdot
2\partial_{x_j}(f^{-1})c(dx_n)\eta_n+c(dx_n)\sum_jc(dx_j)\cdot
2\partial_{x_j}(f^{-1})c(dx_n)\eta^{2}_n\bigg]\nonumber\\
&\times\bigg((\eta_n+i)^{-2}
(\xi_n-\eta_n)^{-1}\bigg)\bigg]'
\Bigg\}|_{\eta_n=i}\nonumber\\
=&-\frac{i\xi_n+2}{4(\xi_n-i)^2}c(\xi')\sum_jc(dx_j)\cdot
2\partial_{x_j}(f^{-1})c(\xi')
-\frac{i}{4(\xi_n-i)^2}\bigg[c(dx_n)\sum_jc(dx_j)\cdot
2\partial_{x_j}(f^{-1})c(\xi')\nonumber\\
&+c(\xi')\sum_jc(dx_j)\cdot
2\partial_{x_j}(f^{-1})c(dx_n)\bigg]-\frac{i\xi_n}{4(\xi_n-i)^2}c(dx_n)\sum_jc(dx_j)\cdot
2\partial_{x_j}(f^{-1})c(dx_n),
\end{align}
we get
\begin{align}\label{93}
&{\rm trace} \bigg[\pi^+_{\xi_n}\bigg(\frac{c(\xi)\sum_jc(dx_j)\cdot
2\partial_{x_j}(f^{-1})c(\xi)}{|\xi|^4}\bigg)\times
\partial_{\xi_n}\sigma_{-1}(\widetilde{D})^{-1}\bigg](x_0)|_{|\xi'|=1}\nonumber\\
=&\frac{i}{f(1+\xi^2_n)^2}{\rm trace} [c(dx_n)\sum_jc(dx_j)\cdot
2\partial_{x_j}(f^{-1})]+\frac{1}{f(1+\xi^2_n)^2}{\rm trace} [c(\xi')\sum_jc(dx_j)\cdot
2\partial_{x_j}(f^{-1})].
\end{align}
Considering, for $i<n$,
\begin{align}\label{94}
&-i\int_{|\xi'|=1}\int^{+\infty}_{-\infty}{\rm trace}
\bigg[\pi^+_{\xi_n}\bigg(\frac{c(\xi)c(df)c(\xi)}{|\xi|^4}\bigg)\times
\partial_{\xi_n}\sigma_{-1}(\widetilde{D})^{-1}\bigg](x_0)|_{|\xi'|=1}d\xi_n\sigma(\xi')dx'\nonumber\\
=&\frac{\pi}{4}\times\frac{2}{f}{\rm trace}[c(dx_n)\sum_jc(dx_j)\cdot
2\partial_{x_j}(f^{-1})]\Omega_3dx'-\frac{\pi i}{4}\times\frac{2}{f}{\rm trace}[c(\xi')\sum_jc(dx_j)\cdot
2\partial_{x_j}(f^{-1})]\Omega_3dx'\nonumber\\
=&\frac{\pi}{2f}{\rm trace}[c(dx_n)\sum_jc(dx_j)\cdot
2\partial_{x_j}(f^{-1})]\Omega_3dx'-\frac{\pi i}{2f}{\rm trace}[c(\xi')\sum_jc(dx_j)\cdot
2\partial_{x_j}(f^{-1})]\Omega_3dx'.
\end{align}
Then
\begin{align}\label{94}
{\bf case~(b)}=&\frac{9}{2f^{2}}\pi h'(0)\Omega_3dx'+\frac{2\pi}{f^{3}}{\rm trace}[c(dx_n)c(df)]\Omega_3dx'+\frac{\pi}{2f}{\rm trace}[c(dx_n)\sum_jc(dx_j)\cdot
2\partial_{x_j}(f^{-1})]\Omega_3dx'\nonumber\\
&-\frac{\pi i}{2f}{\rm trace}[c(\xi')\sum_jc(dx_j)\cdot
2\partial_{x_j}(f^{-1})]\Omega_3dx'.
\end{align}

\noindent {\bf  case (c)}~$r=-1,~l=-2,~k=j=|\alpha|=0$.\\

\noindent By (\ref{b2}), we get
\begin{align}\label{61}
{\bf case~(c)}=-i\int_{|\xi'|=1}\int^{+\infty}_{-\infty}{\rm trace} [\pi^+_{\xi_n}\sigma_{-1}(\widetilde{D})^{-1}\times
\partial_{\xi_n}\sigma_{-2}(\widetilde{D})^{-1}](x_0)d\xi_n\sigma(\xi')dx'.
\end{align}
By Lemma \ref{le:4.1}, we have
\begin{align}\label{62}
\pi^+_{\xi_n}\sigma_{-1}(\widetilde{D})^{-1}|_{|\xi'|=1}(x_0)=\frac{c(\xi')+ic(dx_n)}{f(\xi_n-i)}.
\end{align}
By (\ref{z:1.1}), we have
\begin{align}\label{65}
&\partial_{\xi_n}\sigma_{-2}(\widetilde{D})^{-1}(x_0)|_{|\xi'|=1}\nonumber\\
=&\partial_{\xi_n}\bigg\{\frac{2}{f}\sigma_{-2}(D)^{-1}+\frac{4}{f^{2}}
\cdot\frac{c(\xi)c(df)c(\xi)}{|\xi|^4}
+\frac{c(\xi)\sum_jc(dx_j)\cdot2\partial_{x_j}(f^{-1})c(\xi)}{|\xi|^4}\bigg\}
\nonumber\\
=&\partial_{\xi_n}\bigg\{\frac{2}{f}\sigma_{-2}(D)^{-1}\bigg\}+\partial_{\xi_n}\bigg(\frac{4}{f^{2}}
\cdot\frac{c(\xi)c(df)c(\xi)}{|\xi|^4}\bigg)+\partial_{\xi_n}\bigg( \frac{c(\xi)\sum_jc(dx_j)\cdot2\partial_{x_j}(f^{-1})c(\xi)}{|\xi|^4}\bigg).
\end{align}
By computations, we have
\begin{align}\label{67}
&\partial_{\xi_n}\bigg(\frac{c(\xi)c(df)c(\xi)}{|\xi|^4}\bigg)(x_0)\nonumber\\
=&-\frac{4\xi_n}{(1+\xi_n^2)^3}c(\xi')c(df)c(\xi')
+\bigg(\frac{1}{(1+\xi_n^2)^2}-\frac{4\xi_n^2}{(1+\xi_n^2)^3}\bigg)\bigg(c(\xi')c(df)c(dx_n)\nonumber\\
&+c(dx_n)c(df)c(\xi')\bigg)+\bigg(\frac{2\xi_n}{(1+\xi_n^2)^2}
-\frac{4\xi_n^3}{(1+\xi_n^2)^3}\bigg)c(dx_n)c(df)c(dx_n).
\end{align}
We denote $$\sigma_{-2}(D^{-1})=\frac{c(\xi)\sigma_0(D)c(\xi)}{|\xi|^4}+\frac{c(\xi)}{|\xi|^6}c(dx_n)[\partial_{x_n}[c(\xi')](x_0)|\xi|^2-c(\xi)h'(0)],$$ then
\begin{align}\label{68}
&\partial_{\xi_n}\big(\sigma_{-2}(D^{-1})\big)(x_0)\nonumber\\
=&\frac{1}{(1+\xi_n^2)^3}\bigg[(2\xi_n-2\xi_n^3)c(dx_n)\sigma_0(D)(x_0)c(dx_n)
+(1-3\xi_n^2)c(dx_n)\sigma_0(D)(x_0)c(\xi')\nonumber\\
&+(1-3\xi_n^2)c(\xi')\sigma_0(D)(x_0)c(dx_n)
-4\xi_nc(\xi')\sigma_0(D)(x_0)c(\xi')
+(3\xi_n^2-1){\partial}_{x_n}c(\xi')\nonumber\\
&-4\xi_nc(\xi')c(dx_n){\partial}_{x_n}c(\xi')
+2h'(0)c(\xi')+2h'(0)\xi_nc(dx_n)\bigg]+6\xi_nh'(0)\frac{c(\xi)c(dx_n)c(\xi)}{(1+\xi^2_n)^4}.
\end{align}
By (\ref{62}) and (\ref{68}), we have
\begin{align}\label{71}
&{\rm tr}[\pi^+_{\xi_n}\sigma_{-1}(\widetilde{D})^{-1}\times
\partial_{\xi_n}\big(\frac{2}{f}\sigma_{-2}(D^{-1})\big)](x_0)\nonumber\\
=&\frac{4}{f^{2}}\cdot\frac{3h'(0)(i\xi^2_n+\xi_n-2i)}{(\xi-i)^3(\xi+i)^3}
+\frac{4}{f^{2}}\cdot\frac{12h'(0)i\xi_n}{(\xi-i)^3(\xi+i)^4}\nonumber\\
=&\frac{12h'(0)(i\xi^2_n+\xi_n-2i)}{f^{2}(\xi-i)^3(\xi+i)^3}
+\frac{48h'(0)i\xi_n}{f^{2}(\xi-i)^3(\xi+i)^4}.
\end{align}
Then
\begin{align}\label{72}
-i\Omega_3\int_{\Gamma_+}\frac{4}{f^{2}}\cdot\bigg[\frac{3h'(0)(i\xi_n^2+\xi_n-2i)}
{(\xi_n-i)^3(\xi_n+i)^3}+\frac{12h'(0)i\xi_n}{(\xi_n-i)^3(\xi_n+i)^4}\bigg]d\xi_ndx'=
-\frac{9}{2f^{2}}\pi h'(0)\Omega_3dx'.
\end{align}
By (\ref{62}) and (\ref{67}), we have
\begin{align}\label{73}
&{\rm tr}[\pi^+_{\xi_n}\sigma_{-1}(\widetilde{D})^{-1}\times
\partial_{\xi_n}\bigg(\frac{c(\xi)c(df)c(\xi)}
{|\xi|^4}\bigg)](x_0)\nonumber\\
=&\frac{2}{f}\cdot\frac{1-3\xi_n^2+3i\xi_n-i\xi_n^3}{(\xi_n-i)^4(\xi_n+i)^3}
{\rm trace}[c(dx_n)c(df)]+\frac{2}{f}\cdot\frac{-i(1-3\xi_n^2)-3\xi_n+\xi_n^3}{(\xi_n-i)^4(\xi_n+i)^3}
{\rm trace}[c(\xi')c(df)].
\end{align}
By $i<n,~\int_{|\xi'|=1}\xi_{i_{1}}\xi_{i_{2}}\cdots\xi_{i_{2d+1}}\sigma(\xi')=0$ and ${\rm trace}~[c(\xi')c(df)]$ has no contribution for computing {\bf case~(c)}, we have
\begin{align}\label{74}
&-i\int_{|\xi'|=1}\int^{+\infty}_{-\infty}{\rm tr}[\pi^+_{\xi_n}\sigma_{-1}(\widetilde{D})^{-1}\times
\partial_{\xi_n}\bigg[\frac{4}{f^{2}}
\cdot\frac{c(\xi)c(df)c(\xi)}
{|\xi|^4}\bigg]](x_0)d\xi_n\sigma(\xi')dx'\nonumber\\
&=-i\int_{|\xi'|=1}\int^{+\infty}_{-\infty}\frac{4}{f^{2}}
\cdot\frac{2}{f}\times\frac{1-3\xi_n^2+3i\xi_n-i\xi_n^3}{(\xi_n-i)^4(\xi_n+i)^3}{\rm trace}[c(dx_n)c(df)]d\xi_n\sigma(\xi')dx'\nonumber\\
&=-\frac{4}{f^{2}}
\cdot\frac{2}{f}\times i\Omega_3{\rm trace}[c(dx_n)c(df)]\int_{\Gamma^+}\frac{1-3\xi_n^2+3i\xi_n-i\xi_n^3}{(\xi_n-i)^4(\xi_n+i)^3}d\xi_ndx'\nonumber\\
&=-\frac{4}{f^{2}}
\cdot\frac{2}{f}\times i\Omega_3{\rm trace}[c(dx_n)c(df)]\frac{2\pi i}{3!}\bigg[\frac{1-3\xi_n^2+3i\xi_n-i\xi_n^3}{(\xi_n+i)^3}\bigg]^{(3)}\bigg|_{\xi_n=i}dx'\nonumber\\
&=-\frac{4}{f^{2}}
\cdot\frac{\pi}{2f}{\rm trace}[c(dx_n)c(df)]\Omega_3dx'\nonumber\\
&=-\frac{2\pi}{f^{3}}
{\rm trace}[c(dx_n)c(df)]\Omega_3dx'.
\end{align}
On the other hand, we get
\begin{align}\label{67}
&\partial_{\xi_n}\bigg(\frac{c(\xi)\sum_jc(dx_j)\cdot2\partial_{x_j}(f^{-1})c(\xi)}
{|\xi|^4}\bigg)(x_0)\nonumber\\
=&-\frac{4\xi_n}{(1+\xi_n^2)^3}c(\xi')\sum_jc(dx_j)\cdot2\partial_{x_j}(f^{-1})c(\xi')
+\bigg(\frac{1}{(1+\xi_n^2)^2}-\frac{4\xi_n^2}{(1+\xi_n^2)^3}\bigg)\nonumber\\
&\times\bigg(c(\xi')
\sum_jc(dx_j)\cdot2\partial_{x_j}(f^{-1})c(dx_n)+c(dx_n)
\sum_jc(dx_j)\cdot2\partial_{x_j}(f^{-1})c(\xi')\bigg)\nonumber\\
&
+\bigg(\frac{2\xi_n}{(1+\xi_n^2)^2}
-\frac{4\xi_n^3}{(1+\xi_n^2)^3}\bigg)c(dx_n)\sum_jc(dx_j)\cdot2\partial_{x_j}(f^{-1})c(dx_n).
\end{align}
By (\ref{62}), we have
\begin{align}\label{73}
&{\rm tr}[\pi^+_{\xi_n}\sigma_{-1}(\widetilde{D})^{-1}\times
\partial_{\xi_n}\bigg(\frac{c(\xi)\sum_jc(dx_j)\cdot2\partial_{x_j}(f^{-1})c(\xi)}
{|\xi|^4}\bigg)](x_0)\nonumber\\
=&\frac{2}{f}\cdot\frac{1-3\xi_n^2+3i\xi_n-i\xi_n^3}{(\xi_n-i)^4(\xi_n+i)^3}
{\rm trace}[c(dx_n)\sum_jc(dx_j)\cdot2\partial_{x_j}(f^{-1})]\nonumber\\
&+\frac{2}{f}\cdot\frac{i(1-3\xi_n^2)-3\xi_n+\xi_n^3}{(\xi_n-i)^4(\xi_n+i)^3}
{\rm trace}[c(\xi')\sum_jc(dx_j)\cdot2\partial_{x_j}(f^{-1})].
\end{align}
Then, we have
\begin{align}\label{74}
&-i\int_{|\xi'|=1}\int^{+\infty}_{-\infty}{\rm tr}[\pi^+_{\xi_n}\sigma_{-1}(\widetilde{D})^{-1}\times
\partial_{\xi_n}\bigg[\frac{c(\xi)\sum_jc(dx_j)\cdot2\partial_{x_j}(f^{-1})c(\xi)}
{|\xi|^4}\bigg]](x_0)d\xi_n\sigma(\xi')dx'\nonumber\\
=&-\frac{\pi}{2f}{\rm trace}[c(dx_n)\sum_jc(dx_j)\cdot2\partial_{x_j}(f^{-1})]\Omega_3dx'+\frac{\pi i}{2f}{\rm trace}[c(\xi')\sum_jc(dx_j)\cdot2\partial_{x_j}(f^{-1})]\Omega_3dx'.
\end{align}
Then, we have
\begin{align}\label{94}
{\bf case (c)}=&-\frac{9}{2f^{2}}\pi h'(0)\Omega_3dx'-\frac{2\pi}{f^{3}}
{\rm trace}[c(dx_n)c(df)]\Omega_3dx'-\frac{\pi}{2f}{\rm trace}[c(dx_n)\sum_jc(dx_j)\cdot2\partial_{x_j}(f^{-1})]\Omega_3dx'\nonumber\\
&+\frac{\pi i}{2f}{\rm trace}[c(\xi')\sum_jc(dx_j)\cdot2\partial_{x_j}(f^{-1})]\Omega_3dx'.
\end{align}
Now $\Phi$ is the sum of the {\bf case (a) (I)}-{\bf case (c)}. Therefore, we get
\begin{align}\label{76}
\Phi=\frac{9i-6}{2} f^{-1}\partial_{x_{n}}(f^{-1})\Omega_3dx'.
\end{align}
Hence we conclude that
\begin{thm}\label{999}
Let $M$ be a 4-dimensional compact spin manifold with boundary $\partial M$ and the metric $g^M$ (see (1.3) in \cite{Wa3}). Let $f$ be a nonzero smooth function on $M$ and $\widetilde{D}=\frac{1}{2}fD+c(df)$ be perturbations of Dirac operators, then
 \begin{align}
\widetilde{{\rm Wres}}[\pi^{+}\widetilde{D}^{-1} \circ\pi^{+}\widetilde{D}^{-1}]
=&
\int_{M}\frac{-512\pi^{2}}{f^{2}}\bigg\{\frac{1}{12}s+\frac{\Delta f}{2f}
+\frac{1}{2}g(df,df^{-1})+\frac{2|df|^{2}}{f^{2}}\bigg\}d{\rm vol_{M} }\nonumber\\
&+\int_{\partial M}\Bigg\{\frac{9i-6}{2} f^{-1}\partial_{x_{n}}(f^{-1})\Omega_3dx'\Bigg\}d{\rm vol_{\partial M}},
\end{align}
where  $s$ is the scalar curvature.
\end{thm}

\section*{Acknowledgements}
This work was supported by NSFC No.12301063 and NSFC No.11771070 and Basic research Project of the Education Department of Liaoning Province (Grant No. LJKQZ20222442). The authors thank the referee for his (or her) careful reading and helpful comments.

\end{document}